\documentclass[pdflatex,sn-mathphys-num]{sn-jnl}

\usepackage{graphicx}%
\usepackage{multirow}%
\usepackage{amsmath,amssymb,amsfonts}%
\usepackage{amsthm}%
\usepackage{mathrsfs}%
\usepackage[title]{appendix}%
\usepackage{xcolor}%
\usepackage{textcomp}%
\usepackage{manyfoot}%
\usepackage{booktabs}%
\usepackage{algorithm}%
\usepackage{algorithmicx}%
\usepackage{algpseudocode}%
\usepackage{listings}%

\theoremstyle{thmstyleone}%
\theoremstyle{thmstyletwo}%

\theoremstyle{thmstylethree}%

\newtheorem{thm}{\bf Theorem}[section]
\newtheorem{cor}[thm]{\bf Corollary}
\newtheorem{lem}[thm]{\bf Lemma}

\newtheorem{rem}[thm]{\bf Remark}
\newtheorem{exam}[thm]{\bf Example}

\def\zR{\mathbb R}

\begin{document}

\title[Calculating  eigenvectors in max-times algebra]{Calculating eigenvectors in max-times algebra by mutation-sunflower method}


\author*[1]{\fnm{Seyed Mahmoud } \sur{Manjegani}}\email{manjgani@iut.ac.ir, manjeg1s@uregina.ca}

\author[2]{\fnm{Aljo\v sa } \sur{Peperko}}\email{aljosa.peperko@fs.uni-lj.si}
\equalcont{These authors contributed equally to this work.}

\author[3]{\fnm{Hojr} \sur{Shokooh Saljooghi}}\email{h.shokooh@math.iut.ac.ir}
\equalcont{These authors contributed equally to this work.}

\affil*[1, 3]{\orgdiv{Department of Mathematical Sciences}, \orgname{Isfahan University of Technology}, \orgaddress{ \city{Isfahan}, \postcode{845683111}, \country{Iran}}}

\affil[1]{\orgdiv{Department of Mathematics and Statistics}, \orgname{University of Regina}, \orgaddress{\street{3737 Wascana Parkway}, \city{Regina, Saskatchewan}, \postcode{S4S 0A2}, \country{canada}}}

\affil[2]{\orgdiv{Faculty of Mechanical Engineering}, \orgname{University of Ljubljana}, \orgaddress{\street{A\v{s}ker\v{c}eva 6}, \city{Ljubljana}, \postcode{SI-1000}, \country{Slovenia}}}

\affil[2]{\orgdiv{Institute of Mathematics, Physics and Mechanics}, \orgname{University of Ljubljana}, \orgaddress{\street{Jadranska 19}, \city{Ljubljana}, \postcode{SI-1000}, \country{Slovenia}}}

\abstract{In this article we introduce a new method, which we call a mutation-sunflower method,
for calculating max-eigenvectors of a nonnegative irreducible $n\times n$ matrix $A$.
Our method works in the general irreducible case; however, its practical usefulness is limited to
some special classes of matrices.
Our method reduces to solving max-eigenproblems for simple mutation-sunflower matrices that have exactly one positive entry in each row.
We include some instructive examples.}



\maketitle

\section{Introduction}
The max-times system is one of the analogues of linear algebra that has recently attracted the attention of many researchers.
The {\it max-times system} consists of nonnegative real numbers $\mathbb{R}_+$ equipped with the following operations: multiplication $a\otimes b=ab$, and max-times algebra summation $a\oplus b=\max\{a,b\}$. Furthermore, max-times algebra is isomorphic to 
{\it max-plus algebra}, which consists of the set $\mathbb{R}\cup\{-\infty\}$ with operations of addition and maximization \cite{Bapat-Stan}, 
and it is also isomorphic to the {\it min-plus algebra}, consisting of the set $\mathbb{R}\cup\{\infty\}$ with operations of addition and minimization
\cite{Bu, litvinor}. This algebra system and its isomorphic versions enable us to study in a linear manner classically non-linear phenomena from different areas such
as parallel computation, transportation networks, timetabled programs, IT, dynamic systems, combinatorial optimization, computational biology,
graph theory, and mathematical physics. For more details, we refer to references
\cite{Bu, Esn-Dri1, Esn-Dri2, Heid, litvinor, KMNS21}. 

In this algebraic system, many problems that appear in linear algebra, like systems of equations, eigenvalues,
projections, subspaces, singular value decomposition, and duality theory, have been developed and have reached other areas such as functional analysis
and combinatorial optimization \cite{Bapat, Bu, litvinor, MPS22}. It is well-known that eigenvectors and eigenvalues for linear systems in max-plus algebra 
play an important role in the study of associated discrete-event systems (see e.g. \cite{Bu, Heid}) and their (ultimately) periodic behaviour.

The article is organized as follows. In Section 2 we present some preliminaries needed for our proofs.
In Section 3 we present our main results by introducing a new mutation-sunflower method for calculating a basis of the principal max-times algebraic eigencone of an
irreducible matrix $A$. Our method works in the general irreducible case, but it is most effective for some special classes of irreducible matrices — for example, matrices of small dimensions, matrices with prescribed (priorly known) circuits, matrices that have a critical circuit of length $n$, and sparse matrices with a small number of circuits.
The method reduces the calculation to solving max-eigenproblems for simple mutation-sunflower matrices that have exactly one positive
entry in each row. We illustrate the method  by several examples.

\section{Preliminaries}

For a $n\times n$ matrix $A$ the conventional eigenequation for eigenvalue $\lambda$ and corresponding eigenvector $x$ is $Ax=\lambda x$,
 $x \neq 0$. In the max-times system the eigenequation for a nonnegative matrix $A=[a_{ij}]=[a_{i,j}]$ is $ A\otimes x=\lambda
x$, where $(A\otimes
x)_{i}=\max_{j=1, \ldots , n}(a_{ij}x_{j}),~x=(x_{1},x_{2},\dots,x_{n}) \in  \mathbb{R}^n _+$, $x\neq 0$ and
$\lambda $ is a (geometric) max-eigenvalue corresponding to max-eigenvector $x$. 
For $A=[a_{ij}], B=[b_{ij}] \in  M_{n}(\zR_{+})=\mathbb{R}_+ ^{n\times n}$ we denote $A\ge B$ if $a_{ij} \ge b_{ij}$ for all $i,j=1, \ldots , n$.
 Similarly, for $x,y \in \mathbb{R}^n _+ $ we write $x\ge y$ if $x_i \ge y_i$ for all $i=1, \ldots , n$. The $ij$th coordinate of the matrix product $A\otimes B$ in max-times algebra is defined with 
 $[A\otimes B]_{ij}= \max_{k=1,\ldots , n}a_{ik} b_{kj}$ for $i,j\in \{1,\ldots , n\}$. The $k$th power of $A$ in max-times algebra is denoted  by $A^k _{\otimes}$.

We call a matrix
$A\in M_{n}(\zR_{+}) $  
reducible
 if either $n=1$
and $A=0$ or if $ n\ge 2 $, there exist a permutation matrix $
P\in M_{n} $ and an integer  $r$ with $ 1\le r
\le n-1 $ such that
\[
P^TAP=\begin{pmatrix}
B & C \\ 0 & D
\end{pmatrix},
\]
where $ B \in M_{r}(\zR_{+}) $, $ D \in M_{n-r}(\zR_{+}) $, $ C\in M_{r\times {(n-r)}}(\zR_{+}) $, and $ 0\in M_{(n-r)\times r}(\zR) $ is a zero
matrix.
 Matrix $ A\in M_{n}(\zR_{+})$ is said to be irreducible if it is not
reducible. 

Let
 $A=[a_{ij}] \in M_{n}(\zR_{+})$.  The weighted directed graph associated
 with $A$ is denoted by $D(A)=(V,E)$, where  $A$ has vertex set $V=\{1,2,\dots,n \}$ and  edge $(i,j)$ from $i$ to $j$
 with weight $a_{ij}$ if and only if $a_{ij}>0$. A circuit (or cycle) of length $k$ is a sequence of $k$ edges
  $(i_{1},i_{2}),\dots,(i_{k},i_{1})$, where $i_{1},i_{2},\dots,i_{k}$ are distinct (a circuit $(i,i)$ of length one is called a {\it loop} 
 \cite{Hor-Joh1}). The $k^{th}$ positive root
 of {\it circuit product} $a_{i_1,i_2}\dots a_{i_k,i_1} $ is called a {\it circuit geometric mean} of matrix $A$.
 The maximum circuit geometric mean in $D(A)$ is denoted by $\mu(A)$.  It is known that $\mu (A)$ is the largest max-eigenvalue of $A$. This concept has been recently addressed in the literature in the context of the symmetrization of the max-times algebra \cite{Briec}. In that work, the connection with the max-times permanent is explicitly established, and the relationship between matrix circuits and the permanent is derived (see Equation (6.6), p. 204). 
 A circuit with circuit geometric mean equal to $\mu(A)$
  is called a {\it critical circuit}, and vertices on critical circuits are called {\it critical vertices}.
  Assuming that simultaneous row and column permutations have been performed on
   $A$ so that the critical vertices are in the leading rows and columns, the {\it critical matrix}
    of A, denoted by $  A^{c}=[a_{ij}^c]$, is formed from the principal submatrix of $A$ on the rows and columns corresponding to critical vertices by
 setting 
\begin{equation*}
{a_{ij}}^{c} = \left\{
\begin{array}{rl}
a_{ij}  &  ~~\text{  if $(i,j)$ is in a critical circuit},~~~  \\
0  & ~~\text{  otherwise.  }
\end{array} \right.
\end{equation*}
Thus the
critical graph $C(A)=D(A^{c}) $ has vertex set $N_c(A)=V^{c}$ of all critical
vertices \cite{Esn-Dri1}. 

A directed graph $ D(A) $ is {\it strongly connected} if
between every pair of distinct nodes $ p_{i},p_{j} $ in $ D(A) $
there is a directed path of finite length that begins at $ p_{i} $
and ends $ p_{j} $ (see e.g. \cite{Hor-Joh1}).  A strongly connected subgraph is called maximal if it can not be included as a subgraph in a larger strongly connected subgraph.
A maximal strongly connected subgraph of $C(A)$ is called a strongly connected component of $C(A)$ (\cite{BHSST}).
\begin{thm}\cite{Hor-Joh1}\label{thm}
Let $ A \in M_{n} , n\ge 2$. The following are equivalent:
\begin{enumerate}
\item A is irreducible.
\item $ D(A)$is strongly connected.
\end{enumerate}
\end{thm}

\noindent 
An important task in max-times algebra is to calculate $\mu(A)$ and several methods may be used to calculate it \cite{Bu, Esn-Dri1, Esn-Dri2}. 
The (geometric) spectrum (the set of (geometric) max-eigenvalues)  of a nonnegative $n\times n$ matrix $A$ in  max-times algebra is denoted by $\sigma_{\otimes}(A)$. 

If $A$ is irreducible, then the max-spectrum of $A$ has a unique element $\mu(A)$
and all corresponding max-eigenvectors are strictly positive (\cite{Bapat, Bapat-Stan, Bu}). More details are found in the following result.

\begin{thm}\label{thm1}  (\cite{Bapat, Bapat-Stan, Bu})
Let $ A \in M_{n} $ be a nonnegative irreducible matrix. Then $\mu(A)$  is positive, it is the unique max-eigenvalue and there exists a positive max-eigenvector x, such that $A\otimes x=\mu(A) x$. 
This max-eigenvector is unique (up to scalar multiples) if and only if the graph of $A^{c}$ (the critical graph of $A$) is strongly connected.
\end{thm}

We now present some elements of the theory of max-linear cones, independency and bases in max-times algebra, following \cite{Bu}.  Let $S\subset\zR_+^n$. The set $S$ is called a max-times algebraic cone (or max-cone) if 
\[
\alpha u\oplus\beta v\in S
\]
for every $u, v \in S$ and $\alpha, \beta\in \zR_+$. This notion has appeared previously in the literature under the name $\mathbb{B}$-space (see \cite{Briec-Horvath}); here we use the term max-cone. A vector $v=(v_1,v_2,\dots, v_n)\in\zR_+^n$ is called a max-combination of $S$ if
\[
v=\sum_{x\in S}^{\oplus}\alpha_x x, \quad \alpha_x\in\zR_+,
\]
where only a finite number of $\alpha_x$ are non-zero. The set of all max-combinations of $S$ is denoted by span$(S)$. We set span$(\emptyset)=\{0\}$. It is clear that span$(S)$ is a max-times algebraic cone. If span$(S)=T$, then $S$ is called the set of generators for $T$.

A vector $v\in S$ is called an {\it extremal} in $S$ if $v=u\oplus w$ for $u, w\in S$ implies $v=u$ or $v=w$. Clearly, if $v\in S$ is an extremal in $S$,  then the points  from $\{\alpha v: \alpha\in\zR_+ \}$ are also extreme in $S$ and form an extremal ray of $S$.   

Let $v=(v_1,v_2,\dots, v_n)\in\zR_+^n, v\neq 0$. 
Then $v$ is called scaled if \[\|v\|_{\infty}= \max(v_1,v_2,\dots, v_n)= 1.\] The set $S$ is called scaled if all its elements are scaled.

The set $S$ is called {\it dependent} if $v$ is a max-combination of $S-\{v\}$ for some $v\in S$. Otherwise $S$ is independent. Let $S, W\subseteq\zR_+^n.$ The set $S$  is called a basis of $W$ if it is an independent set of generators for $W$. 

Let $V_{\oplus} (A, \mu (A))$ be a principal max-eigencone  of $A\in \mathbb{R}^{n\times n} _{+}$ (see e.g.  \cite[Chapters 3 and 4]{Bu}), i.e., 
$V_{\oplus}(A, \mu (A))= \{x \in \mathbb{R}^n _+ : A \otimes x = \mu (A)x \}$. Clearly, $V_{\oplus}(A, \mu(A))$ is a subset of $\zR_+^n$ closed under max-addition and nonnegative scalar multiplication, that is, it is a max-cone of $\zR_+^n$. Recall that for an irreducible matrix $A$ it may happen that the basis  
of a principal max-eigencone  
may be of cardinality larger than one and this cardinality corresponds to a number of nonequivalent critical nodes (e.g., to the number of components of connectivity of a critical graph; see e.g. Theorem 4.3.5 of \cite{Bu}). We also recall two results from \cite{Bu} (\cite[Lemma 3.3.1, Theorem 3.3.9 and Corollary 3.3.11]{Bu}).
\begin{lem}\label{generator1} Let $S$ be a set of scaled generators of a max-times algebraic cone $T\subseteq\zR_+^n$ and let $v$ be a scaled extremal in $T$. Then $v\in S$.
\end{lem}
\begin{thm}\label{generator2} The set of scaled extremals of a max-times algebraic cone $T$ is a basis of $T$ and it is a minimal set of generators for $T$. If $T$ is finitely generated,  then the set of scaled extremals is non-empty and it is a unique scaled basis for $T$.
\end{thm}
 
 The Kleene star matrix  for a nonnegative matrix $A$ is equal to
 \[
\Delta(A)=I\oplus A\oplus A^{2} _{\otimes }\oplus\cdots \oplus A^{ n-1} _{\otimes },
 \] when $\mu(A)\le 1$ (\cite{Bu}). 

  We rewrite \cite[Theorem 6.2]{BHSST} (see also \cite[Theorem 4.3.3]{Bu}) with our notation in the following theorem. 
 \begin{thm} Let $A\in  \mathbb{R}^{n\times n}_{+}$ be an irreducible matrix with $\mu(A)=1$ and let $r$ denote the number of strongly connected components of $C(A)$.  Then: 
 \begin{enumerate}
 \item Each strongly connected component $C_s (1\le s\le r)$ of $C(A)$ corresponds to a max-eigenvector defined as the $i$th column $\Delta(A)$ with $i\in N_{C_s}$, all such columns with $i\in N_{C_s}$ being multiple of each other.
 \item $V_{\oplus}(A, \mu(A))$ is generated by columns of $\Delta(A)$ that are max-eigenvectors of $A$.
 \item Vectors in  the set $\{\Delta(A)_{i_1},\cdots, \Delta(A)_{i_r}\}$, where we take exactly one column of $\Delta(A)$ for each strongly connected component $C_s (1\le s\le r)$ are extremals in $V_{\oplus}(A, \mu(A))$ and form  a basis of $V_{\oplus}(A, \mu(A))$.
 \end{enumerate}
 \label{eig_components}
 \end{thm}
 
To our knowledge the most efficient known algorithm with known complexity for calculating $\mu (A)$ is Karp's algorithm (see e.g. \cite{Bu, Heid}).
In \cite{Esn-Dri1}
a power method algorithm is given to compute the max-eigenvalue $\mu(A)$ and max-eigenvector $x$ of an irreducible nonnegative $n\times n$ matrix $A$.  
There are some limitations of the power method. 
How to select  the primary vector $x(0)$ is one of these limitations 
since it affects the results. A method of calculating  a basis $V_{\oplus}(A, \mu(A))$ based on Theorem \ref{eig_components} has polynomial complexity \cite{Bu}. Howard's algorithm is another in practice very efficient way of calculating max-eigenvectors (see e.g \cite{Heid}). This algorithm utilizes  sunflower matrices (see e.g. \cite[Sections 3.4 and 6.1]{Heid}), which are closely related to our mutation-sunflower matrices that we define in the next section. To be more precise, mutation-sunflower matrices are in fact  special case of sunflower matrices (modulo the ``translation'' to max-plus setting and modulo the reverse edge orientation used in \cite{Heid}). One of the reasons to use the adjective ``mutation'' is that the mutation-sunflower matrices defined below are obtained via suitable changes of the matrix $A$.   The mutation-sunflower method that we introduce calculates  a basis of $V_{\oplus}(A, \mu(A))$ and does not have  polynomial complexity in general, since we need to find all (pairwise disjoint) critical circuits and calculate products of weights of all noncritical circuits. However, it is useful for some special classes of non-negative irreducible $n\times n$ matrices, e.g. 
for the following ones: 
\begin{enumerate}
 \item for sparse enough irreducible matrices with  small number of circuits (in comparison to the size of a matrix),
 \item for irreducible matrices that have a critical circuit of length $n$,
 \item for irreducible matrices of relatively small dimensions,
 \item for irreducible matrices with prescribed  (priorly known) circuits.
\end{enumerate}
Practically, Howard's algorithm  is more efficient than our method. In comparison, besides the case of relatively small size matrices, our method performs well in a special case of matrices, where the circuits of $D(A)$ are prescribed (priorly known). In the following example, we will demonstrate that the mutation-sunflower method requires fewer operations compared to other methods.

\begin{exam}\label{ESM} In \cite{BCOQ}, the matrix related to three railway stations, namely $S_1, S_2$, and $S_3$, is presented. These stations are interconnected by a railway system consisting of two inner circles where trains operate in opposite directions, as well as three outer circles. The outer circles are responsible for passenger transportation to and from local stations. To illustrate a comparison between our methods and others, we have selected specific numbers for $S_{ij}$ and provide the following example.
Let
\[
A=\begin{pmatrix}
 1 & 0 & 2 & 0 &0 &0 &2  & 0 & 0 \\ 
 2 & 1  & 0 & 0 &0 & 0 & 0 &4 & 0 \\ 
 0& 2  & 1 & 0 &0 & 0 &0 & 0 & 2 \\ 
 0& 0 & 0 & 1 & 4 &0 & 2  &0 & 0 \\ 
 0& 0 & 0 & 0 & 1& 4 & 0 & 4 & 0 \\ 
 0& 0 & 0 & 4  & 0 & 1 &0  & 0 & 2\\ 
  0& 0 & 2 & 0 & 4 & 0& 2& 0 & 0 \\ 
2 & 0 &0  & 0 & 0 & 4& 0& 4& 0 \\ 
 0& 2 & 0 & 4 &0 & 0& 0 & 0 & 2 \\ 
\end{pmatrix}
\]
We have compared mutation-sunflower method with Karp's algorithm and power method for this example in the Appendix A.
\end{exam}

\section{Mutation--sunflower method for calculating max--eigenvectors}

Let $A$ be an irreducible matrix and let $r$ be the number of pairwise disjoint critical circuits (two circuits are called pairwise disjoint when they have no common vertices),  
i.e., $r$ equals the number of components of connectivity of the critical graph $C(A)$ of $A$. So there are exactly $r$ max-independent positive max-eigenvectors corresponding to $\mu (A)$ 
 which form a  basis of the principal max-eigencone $V_{\oplus} (A, \mu (A))$ of $A$ (see   Theorem \ref{eig_components}).  
Observe that since $A$ is irreducible then every non-zero $x\in V_{\oplus}(A, \mu(A))$ is (strictly) positive by \cite[Theorem 4.4.8]{Bu}.   

Let us describe our mutation-sunflower method for calculating 
a  basis of $V_{\oplus} (A, \mu (A))$. 
We construct mutation-sunflower matrices $A^* _1, \dots , A^*_{r}$ (defined in the following), which have one non-zero element in each row. 
Therefore it is easy to calculate $\mu (A^* _i)$ for all $i=1, \ldots , r$ and the corresponding max-eigenvectors (see also \cite[Lemma 3.19]{Heid}). Moreover, by construction $\mu (A)=\mu (A^* _i)$ for all $i=1, \ldots , r$  and it turns out that there exists a basis  $\{x_{1},\ldots,x_{r}\}$ of positive vectors
for $V_{\oplus}(A, \mu (A))$,
where each $x_i \in V_{\oplus}(A_{i} ^*, \mu (A))$ 
 (see Theorem \ref{main} below).

Given a nonnegative irreducible $n \times n$ matrix $A$, the matrices  $A^* _1, \ldots , A^*_{r}$ are defined by the following steps:
\medskip

\begin{enumerate}
 \item[Step (1):] Let $k=1$. Calculate $\mu (A)$ (in the following we assume that $\mu(A)=1$, if not consider $A^{'}=\frac{1}{\mu(A)}A$ instead of $A$) and  find  $r $ critical circuits 
$\mathcal{C}$ of $A$, each belonging to a different component of connectivity of $C(A)$.  For each of these $\mathcal{C}$'s do steps (2)-(4). 

\item[Step (2):]  If $\mathcal{C}$ has $n$ nodes, then put the weights of edges of  circuit $\mathcal{C}$ as the corresponding entries in $A^* _k$. The other entries in $A^* _k$ are set to be equal to zero. 

 \noindent Otherwise go to Step (3).

\item[Step (3):]  Arrange all other
noncritical 
circuits (which are not loops) $\mathcal{C}_i$  in $D(A)$ with at least one node from $\mathcal{C}$ 
in  decreasing order with respect to the maximum product $M_i$
of  weights of edges.  
So \[M_1\geq M_2\geq \dots\geq M_m.\]
Finally, set $h=1$. 

\item[Step (4):] First  put the weights $a_{ij}$ of edges $(i,j)$  of $\mathcal{C}$ as the corresponding $ij$th entries in $A^* _k$. The other entries in the rows  of $A^* _k$ that contain these entries are set to be equal to zero. 

 For other rows; choose $\mathcal{C}_h$ with the product $M_h$ (if there are several possibilities for a choice of $\mathcal{C}_h$, choose one of them). Then put the weights of edges from $\mathcal{C}_h$ that lie in not yet determined rows of $A^* _k$ as the corresponding entries in $A^* _k$. The other entries in these rows are set to zero.  
 
 If $A^* _ k$ has $n$ positive entries increase the value of $k$ by $1$. Finally, increase the value of $h$ by $1$ and repeat this step until $h=m+1$. 
\end{enumerate}
\bigskip

\begin{rem} 
\label{what we need}
\noindent (i) All possible matrices $A^* _1, \ldots , A^*_{r}$ defined by steps (1)-(4) are well defined (all $n$ rows of $A^* _k$, $k=1,\ldots,r$, are determined), since $A$ is irreducible. We 
call them the {\bf mutation-sunflower matrices}. 
Each (positive) vector from $V_{\oplus}(A, \mu(A))$ that is  also a max-eigenvector of some mutation-sunflower matrix is called a {\bf principal-sunflower max-eigenvector of $A$}. Observe that since each $A_k ^*$ has only one non zero entry in each row, it has only one independent max-eigenvector.

(ii) 
Observe that there might be several possible collections of  mutation-sunflower matrices. One reason for this is since each component of connectivity of a critical graph might contain several non-disjoint critical circuits and in Step (1) we choose one of them.  The other reason is that in Step (4) for each chosen $\mathcal{C}$ and each $h$ there can be several possibilities for a choice of $\mathcal{C}_h$. In Theorem \ref{main} we however establish that it does not matter which choice we make here, since we show that all possible mutation-sunflower matrices that are associated to critical circuits from the same component of connectivity of $C(A)$ have a common scaled principal-sunflower max-eigenvector, i.e., there exist a  scaled principal-sunflower max-eigenvector that is a max-eigenvector of all of these mutation-sunflower matrices.
 
(iii) In a special case when $A$ has critical circuit $\mathcal{C}$ of lenght $n$, then only steps (1) and (2) are necessary.

\end{rem}
\begin{lem}\label{L11}Let $A \in \mathbb{R}_+ ^{n\times n}$ be an irreducible matrix such that $\mu (A)=1$ and let $\mathcal{C}$ be a given critical circuit of $A$ and let $A^*$ be any mutation-sunflower matrix produced by steps (2) to (4) corresponding to $\mathcal{C}$. For each $j$ which is not a node of $\mathcal{C}$ and each $i$ which is a node of  $\mathcal{C}$, $D(A^*)$ contains a path from $j$ to  $i$ with the greatest product of weights among all paths from $j$ to $i$ in $D(A)$.

\end{lem}
\begin{proof} Suppose that $A^*$ is a mutation-sunflower matrix produced by steps (2) to (4) corresponding to a critical circuit $\mathcal{C}$. By construction of $A^*$, we
put the weights of edges of a circuit $\mathcal{C}$ as the corresponding entries in $A^*$, and then choose the circuit $\mathcal{C}_1$ and put the weights of edges from $\mathcal{C}_1$ that lie in not yet determined rows of $A^*$ as the corresponding entries in $A^*$. The other entries in these rows are set to zero. We continue with the selection of next $M_i$'s $(M_2, M_3,...)$ until every row of $A^*$ has a non zero entry. Since we put the weights of edges from $\mathcal{C}_k$ to the corresponding rows that are not selected from the previous $\mathcal{C}_l$'s ($1\le l\le k-1)$ or from $\mathcal{C}$ and since each $M_k$ is product of weights of  edges in a noncritical circuit $\mathcal{C}_k$ with a node from $\mathcal{C}$, a non zero entry in each row of $A^*$ lies in path with the heaviest weight from the node corresponding to this row to each critical node from $\mathcal{C}$. Indeed, if  all weights from $M_k$ are selected in the definition of $A^*$, then the result follows by definition of $M_k$. If we do not select some weights from $M_k$, this means that their corresponding rows have been filled from previous $M_k$'s or $\mathcal{C}$. 
Therefore, if we multiply the non-zero entry of a fixed row in $A^*$ by the non-zero entries of some rows in $A^*$, we get a path with the greatest product of weights from the node corresponding to this row and to any chosen critical node from $\mathcal{C}$, since $M_1\geq M_2\geq \dots\geq M_m.$ 
\end{proof}
 \begin{thm} \label{main} Let $A\in  \mathbb{R}^{n\times n}_{+}$ be an irreducible matrix with $\mu(A)=1$. Then each column of $\Delta(A)$ that is a max-eigenvector of $A$ is a principal-sunflower max-eigenvector of $A$.
 
 Moreover, all possible mutation-sunflower matrices that are associated to critical circuits from the same component of connectivity of $C(A)$ have a common scaled principal-sunflower max-eigenvector.
 \end{thm}
 \begin{proof} 
Let $g_i$ be the $i$-th column of $\Delta(A)$, where $i\in N_C(A)$, and let $r$ be the number of components of connectivity of $C(A)$ (according to construction of $\Delta(A)$, $g_{ii}=1$). Then there is $1\le s\le r$ such that $i$ is a node in a critical circuit $\mathcal{C}_s$. We show that $g_i$ is a max-eigenvector (and also a classical eigenvector) of a mutation-sunflower matrix $A_s^*$ corresponding to $\mathcal{C}_s$ (since there may be several mutation-sunflower matrices related to $\mathcal{C}_s$, $A_s^*$ could be any of them). Since $A_s^*\le A$, we have $A_s^*\otimes g_i\le A\otimes g_i=g_i$ and therefore the following inequalities hold:
\begin{eqnarray*}
a_{1 j_1} g_{\scriptstyle j_1 i} &\le & g_{1 i} \\
a_{2 j_2} g_{\scriptstyle j_2 i} &\le & g_{2 i} \\
\vdots \\
a_{n j_n} g_{\scriptstyle j_n i} &\le & g_{n i},
\end{eqnarray*}
where $a_{1 j_1}, \ldots , a_{n j_n}$ are positive entries in $A_s^*$ (recall that $A_s^*$ has only one non-zero entry in each row). 

Because some entries in $A_s^*$ have been selected from the critical circuit $\mathcal{C}_s$, let us denote the nodes in $\mathcal{C}_s$ by
\[
j_{\scriptstyle f(1)}, \dots, j_{\scriptstyle f(k)},
\]
where $k\le n$ and $f(k+1) = f(1) = i$ ($f$ is a function from $\{1, 2, \dots, k+1\}$ to $\{1, 2, \dots, n\}$). It is clear that for every $1 \le p \le k$, 
\[
a_{\scriptstyle f(p) j_{f(p)}} g_{\scriptstyle j_{f(p)} i} = g_{\scriptstyle f(p) i} 
\]
(if not, then by multiplying the left and right sides of the inequalities above that include coefficients from $\mathcal{C}_s$, one obtains $\mu(A)<1$, which is a contradiction).

Now let $a_{\scriptstyle l j_l} g_{\scriptstyle j_l i} < g_{\scriptstyle l i}$ for some $l$ which is not a node of $\mathcal{C}_s$ such that $a_{\scriptstyle l j_l}$ is a positive entry in $A_s^*$. We consider two cases;  
case 1: if $j_l = i = f(1)$, then because $g_{ii}=1$ we have $a_{\scriptstyle l i} < g_{\scriptstyle l i}$. Since each row of $A_s^*$ has only one nonzero entry, $a_{\scriptstyle l i}$ is the only nonzero entry in row $l$ of $A_s^*$. Then any product of weights corresponding to a path from node $l$ to node $i$ must contain $a_{\scriptstyle l i}$. In fact, its weight is of the form
\[
a_{\scriptstyle l i} = a_{\scriptstyle l f(1)} a_{\scriptstyle f(1) f(2)} a_{\scriptstyle f(2) f(3)} \cdots a_{\scriptstyle f(k-1) f(k)} a_{\scriptstyle f(k) f(1)}.
\]
Since $f(1), f(2), \dots, f(k)$ are nodes of $\mathcal{C}_s$, we have
\[
a_{\scriptstyle f(1) f(2)} a_{\scriptstyle f(2) f(3)} \cdots a_{\scriptstyle f(k) f(1)} = 1 \quad (\text{because } \mu(A)=1),
\]
so $a_{\scriptstyle l i}$ is the heaviest weight of paths from node $l$ to node $i$.
Case 2: if $j_l \ne i$, because $l$ is not in $\mathcal{C}_s$ and $i$ is, $A_s^*$ contains a path from $l$ to $i$ with the greatest product of weights among all paths from $l$ to $i$ in $D(A)$ by Lemma \ref{L11}. Then multiplying the left and right sides of inequalities that contain a weight from this path gives
\[
a_{\scriptstyle l j_l} a_{\scriptstyle j_l f(r_2)} a_{\scriptstyle f(r_2) f(r_3)} \cdots a_{\scriptstyle f(r_k) i} g_{\scriptstyle j_l i} g_{\scriptstyle f(r_2) i} \cdots g_{\scriptstyle f(r_k) i} g_{ii}
<
g_{\scriptstyle l i} g_{\scriptstyle j_l i} g_{\scriptstyle f(r_2) i} \cdots g_{\scriptstyle f(r_k) i},
\]
and consequently 
\[
a_{\scriptstyle l j_l} a_{\scriptstyle j_l f(1)} \cdots a_{\scriptstyle f(r_k) i} < g_{\scriptstyle l i} \quad (\text{since } g_{ii}=1).
\]
This is a contradiction in both cases, because both sides represent the heaviest weights of paths from node $l$ to node $i$ (by Lemma \ref{L11} and the fact that for $i\neq l$, the $li$-th entry of $\Delta(A)$ equals the greatest weight of an $l-i$ path). Thus $g_i$ is a max-eigenvector of $A_s^*$, hence a principal-sunflower max-eigenvector related to $A_s^*$.

Since $i \in N_C(A)$, $\mathcal{C}_s$ and $A_s^*$ were chosen arbitrarily, it follows from Theorem \ref{eig_components} that all mutation-sunflower matrices associated to critical circuits from the same component of $C(A)$ have a common scaled principal-sunflower max-eigenvector.
\end{proof}

Applying the above result to the matrix $\frac{1}{\mu(A)}A$ we obtain the following result.
\begin{cor}\label{essential P1} Let $A\in  \mathbb{R}^{n\times n}_{+}$ be an irreducible matrix. Then 
a basis of  $V_{\oplus}(A, \mu(A))$ is formed by principal-sunflower max-eigenvectors of $A$.
\end{cor}
\bigskip

We can use the following algorithm to calculate the basis $X$ of the principal max-eigencone of $A$. 
\bigskip

\noindent{\bf Algorithm:} Input: 
let $r$ be the number of strongly connected components of $C(A)$ and let $A^{'}=\frac{1}{\mu(A)}A$.

$X_0 := \emptyset .$

For $i= 1, \ldots , r$:

\indent \indent  Calculate $v_i \in V((A^{'} )_ i ^*, 1)$ and define $X_{i}:= X_{i-1} \cup \{v_i\}$

\indent \indent If $i=r$, stop and return $X=X_i.$
\bigskip

\begin{rem} 
The set $X$ calculated by the above algorithm  is a basis of the principal max-eigencone of $A$ by Corollary \ref{essential P1} and Theorem \ref{eig_components}.
\end{rem}
\medskip

We illustrate our method with the following examples.

\begin{exam}\label{E218}
Let
\[
A=\begin{pmatrix}
1 & 2 & 1\\ 2 & 2 & 1 \\ 1 & 1 & 2
\end{pmatrix}.
\]
Then $r=2$ and
 we have 
\[
\mu(A)=\sqrt{a_{12}a_{21}}=a_{22}=a_{33}=2,\]

\[
A^{'}=\begin{pmatrix}
\frac{1}{2} & 1 & \frac{1}{2}\\ 1 & 1 & \frac{1}{2} \\ \frac{1}{2} & \frac{1}{2} & 1
\end{pmatrix}.
\]
We first consider the critical circuit $(1,2), (2,1)$. Then
\[
M_1=\frac{a_{23}}{2}\frac{a_{31}}{2}\frac{a_{12}}{2}=M_2=\frac{a_{23}}{2}\frac{a_{32}}{2}=M_3=\frac{a_{31}}{2}\frac{a_{13}}{2}=\frac{1}{4}\,.
\]
There are two possible mutation-sunflower matrices corresponding to the critical circuit $(1,2), (2,1)$ are 
\[
A_{1_1}^{*}=\begin{pmatrix}
0 & a_{12} & 0\\ a_{21} & 0 & 0 \\ a_{31} & 0 & 0
\end{pmatrix}=
\begin{pmatrix}
0 & 2 & 0\\ 2 & 0 & 0 \\ 1 & 0 & 0
\end{pmatrix}, \quad
A_{1_2}^{*}=\begin{pmatrix}
0 & a_{12} & 0\\ a_{21} & 0 & 0 \\ 0 & a_{32} & 0
\end{pmatrix}=
\begin{pmatrix}
0 & 2 & 0\\ 2 & 0 & 0 \\ 0 & 1 & 0
\end{pmatrix},
\]
If we choose instead a critical circuit $(2,2)$ (which is nondisjoint from  $(1,2), (2,1)$), we would obtain the alternative mutation-sunflower matrices:
\[
A_{1_3}^{*} =\begin{pmatrix}
0 & a_{12} & 0\\ 0 & a_{22} & 0 \\ a_{31} & 0 & 0
\end{pmatrix}=
\begin{pmatrix}
0 & 2 & 0\\ 0 & 2 & 0 \\ 1 & 0 & 0
\end{pmatrix}, \quad
A_{1_4}^{*} = \begin{pmatrix}
0 & a_{12} & 0\\ 0 & a_{22} & 0 \\ 0 & a_{32} & 0
\end{pmatrix}=
\begin{pmatrix}
0 & 2 & 0\\ 0 & 2 & 0 \\ 0 & 1 & 0
\end{pmatrix}.
\]
Any of the matrices $A_{1_1}^{*}, A_{1_2}^{*}, A_{1_3}^{*}, A_{1_4}^{*}$ can be chosen for $A_{1}^{*}$ and their corresponding principal-sunflower max-eigenvector of $A$ is

\[
x_1=\begin{pmatrix}
2 \\2\\1
\end{pmatrix}. 
\]
Also, by considering a critical circuit $(3,3)$ 
we obtain
\[
A_2^{*}=\begin{pmatrix}
0 & a_{12} & 0\\ 0 & 0 & a_{23} \\ 0 & 0 & a_{33}
\end{pmatrix}=
\begin{pmatrix}
0 & 2 & 0\\ 0 & 0 & 1 \\ 0 & 0 & 2
\end{pmatrix}.
\]
The corresponding principal-sunflower max-eigenvector of $A$

\[
x_2=\begin{pmatrix}
1 \\1\\2
\end{pmatrix}. 
\]
Thus, $X=\{x_1, x_2\}$ is a basis of the principal max-cone $V_{\oplus}(A,2)$. 
Note also that for $A^{'}=\frac{1}{2}A$  
we have
\[
\Delta(A^{'})=I \oplus A^{'} \oplus (A^{'})^2 _{\otimes}  =\begin{pmatrix}
1 & 1 & \frac{1}{2}\\ 1 & 1 & \frac{1}{2} \\ \frac{1}{2}& \frac{1}{2} & 1
\end{pmatrix}.
\]
The first two columns of $\Delta(A^{'})$ are multiples of $x_1$ and the third column is a multiple of $x_2$.
\end{exam}

\begin{rem}  In the above example $x=(1, 1, 1)^T$ is a max-eigenvector of $A$ corresponding to $\mu(A)=2$. 
It holds $x=\frac{1}{2}(x_1\oplus x_2)$, where $x_1$ and $ x_2$ are the above principal-sunflower max-eigenvectors of $A$.
\end{rem}

\begin{exam} 
Let
\[
A=\begin{pmatrix}
1& 3& 4& 1 \\ 
1&2&0&1\\
4&1&1&3\\
5&2&1&2
\end{pmatrix}, \quad
A^{'}=\begin{pmatrix}
\frac{1}{4}& \frac{3}{4}& 1& \frac{1}{4} \\  \\
\frac{1}{4}&\frac{2}{4}&0&\frac{1}{4}\\ \\
1&\frac{1}{4}&\frac{1}{4}&\frac{3}{4}\\ \\
\frac{5}{4}&\frac{2}{4}&\frac{1}{4}&\frac{2}{4}
\end{pmatrix}
\]
Then $r=1$ and
\[\mu (A) =\sqrt{a_{13}a_{31}}=4, \quad M_1=\frac{a_{41}}{4}\frac{a_{14}}{4}=\frac{5}{16}> M_2=\frac{a_{23}}{4}\frac{a_{13}}{4}=\frac{3}{16}\,.
\]
We obtain
\[
{A^* _1}=\begin{pmatrix}
0& 0& 4& 0 \\ 
0&0&0&1\\
4&0&0&0\\
5&0&0&0
\end{pmatrix}.
\]
The corresponding principal-sunflower max-eigenvector of $A$ is
\[
x=\begin{pmatrix}
4 \\ \frac{5}{4}\\4\\5
\end{pmatrix} . 
\]
For $A^{'}$ 
we have
\[
\Delta(A^{'})=\begin{pmatrix}
 1 & \frac{3}{4} & 1 & \frac{3}{4} \\ \\
 \frac{5}{16} & 1 & \frac{5}{16} & \frac{1}{4} \\ \\
 1 & \frac{3}{4} & 1 & \frac{3}{4} \\ \\
 \frac{5}{4} & \frac{15}{16} & \frac{5}{4} & 1 \\ 
\end{pmatrix}.
\]
The first and third columns of $\Delta(A^{'})$ are multiples of $x$.
\end{exam}
\begin{exam}\label{big}
Let
\setcounter{MaxMatrixCols}{15}
\begin{equation*} 
A=\begin{pmatrix}
3 & 4 & 3 & 1 & 0 & 1 & 1 & 1 & 1 & 0 & 1 & 1 & 1 & 0 & 0\\
0 & 2 & 4 & 0 & 1 & 1 & 0 & 0 & 0 & 1 & 0 & 0 &1 & 1 & 0\\
0 & 0 & 1 & 1 & 4 & 0 & 0 & 1 & 0 & 0 & 0 & 2 & 0 & 0 & 0\\
1 & 0 & 0 & 0 & 1 & 1 & 1 & 1 & 0 & 0 & 0 & 0 & 0 & 1 & 0\\
4 & 0 & 1 & 0 & 3 & 1 & 0 & 0 & 1 & 0 & 0 & 0 & 0 & 0 & 0\\
1 & 1 & 2 & 0 & 0 & 1 & 4 & 0 & 0 & 0 & 0 & 1 & 0 & 0 & 0\\
0 & 0 & 0 & 1 & 0 & 0 & 2 & 4 & 1 & 0 & 0 & 0 & 0 & 0 & 3\\
1 & 0 & 0 & 1 & 0 & 0 & 0 & 2 & 4 & 0 & 0 & 0 & 0 & 0 & 0\\
0 & 0 & 0 & 0 & 0 & 0 & 0 & 0 & 3 & 4 & 0 & 0 & 1 & 1 & 1\\
1 & 0 & 0 & 0 & 0 & 0 & 1 & 1 & 0 & 1 & 4 & 0 & 0 & 0 & 0\\
0 & 0 & 1 & 1 & 0 & 4 & 1 & 1 & 0 & 0 & 1 & 1 & 0 & 0 & 0\\
0 & 0 & 0 & 1 & 1 & 1 & 0 & 0 & 0 & 0 & 0 & 3 & 1 & 0 & 4\\
1 & 1 & 0 & 0 & 0 & 0 & 0 & 0 & 0 & 1 & 0 & 4 & 2 & 1 & 0\\
1 & 1 & 0 & 1 & 0 & 0 & 0 & 0 & 0 & 0 & 0 & 1 & 0 & 1 & 2\\
2 & 0 & 0 & 1 & 0 & 0 & 1 & 0 & 1 & 0 & 0 & 0 & 4 & 1 & 3\\
\end{pmatrix} .
\end{equation*}
 Then $r=3$ and 
\[\mu(A)=({a_{12}a_{23}a_{35}a_{51}})^{\frac{1}{4}}= ({a_{6,7}a_{7,8}a_{8,9}a_{9,10}a_{10,11}a_{11,6}})^{\frac{1}{6}}= ({a_{12,15}a_{15,13}a_{13,12}})^{\frac{1}{3}}=4\]

By choosing a critical circuit $(1,2), (2,3), (3,5), (5,1)$ we obtain
\[ 
A_1^{*}=\begin{pmatrix}
0 & 4 & 0 & 0 & 0 & 0 & 0 & 0 & 0 & 0 & 0 & 0 & 0 & 0 & 0\\
0 & 0 & 4 & 0 & 0 & 0 & 0 & 0 & 0 & 0 & 0 & 0 & 0 & 0 & 0\\
0 & 0 & 0 & 0 & 4 & 0 & 0 & 0 & 0 & 0 & 0 & 0 & 0 & 0 & 0\\
0 & 0 & 0 & 0 & 1 & 0 & 0 & 0 & 0 & 0 & 0 & 0 & 0 & 0 & 0\\
4 & 0 & 0 & 0 & 0 & 0 & 0 & 0 & 0 & 0 & 0 & 0 & 0 & 0 & 0\\
0 & 0 & 0 & 0 & 0 & 0 & 4 & 0 & 0 & 0 & 0 & 0 & 0 & 0 & 0\\
0 & 0 & 0 & 0 & 0 & 0 & 0 & 4 & 0 & 0 & 0 & 0 & 0 & 0 & 0\\
0 & 0 & 0 & 0 & 0 & 0 & 0 & 0 & 4 & 0 & 0 & 0 & 0 & 0 & 0\\
0 & 0 & 0 & 0 & 0 & 0 & 0 & 0 & 0 & 4 & 0 & 0 & 0 & 0 & 0\\
0 & 0 & 0 & 0 & 0 & 0 & 0 & 0 & 0 & 0 & 4 & 0 & 0 & 0 & 0\\
0 & 0 & 0 & 0 & 0 & 4 & 0 & 0 & 0 & 0 & 0 & 0 & 0 & 0 & 0\\
0 & 0 & 0 & 0 & 0 & 0 & 0 & 0 & 0 & 0 & 0 & 0 & 0 & 0 & 4\\
0 & 0 & 0 & 0 & 0 & 0 & 0 & 0 & 0 & 0 & 0 & 4 & 0 & 0 & 0\\
0 & 0 & 0 & 0 & 0 & 0 & 0 & 0 & 0 & 0 & 0 & 0 & 0 & 0 & 2\\
2 & 0 & 0 & 0 & 0 & 0 & 0 & 0 & 0 & 0 & 0 & 0 & 4 & 0 & 0\\
\end{pmatrix} .
\]
 The corresponding principal-sunflower max-eigenvector of $A$ is
\[
x_1=\begin{pmatrix}
1&1&1&\frac{1}{4}&1&\frac{1}{2}&\frac{1}{2}&\frac{1}{2}&\frac{1}{2}&\frac{1}{2}&\frac{1}{2}&\frac{1}{2}&\frac{1}{2}&\frac{1}{4}&\frac{1}{2}
\end{pmatrix}^T.
\]
By choosing a critical circuit $(6,7), (7,8), (8,9), (9,10), (10,11), (11,6)$ it follows
\[ 
A_2^{*}=\begin{pmatrix}
0 & 4 & 0 & 0 & 0 & 0 & 0 & 0 & 0 & 0 & 0 & 0 & 0 & 0 & 0\\
0 & 0 & 4 & 0 & 0 & 0 & 0 & 0 & 0 & 0 & 0 & 0 & 0 & 0 & 0\\
0 & 0 & 0 & 0 & 4 & 0 & 0 & 0 & 0 & 0 & 0 & 0 & 0 & 0 & 0\\
0 & 0 & 0 & 0 & 0 & 1& 0 & 0 & 0 & 0 & 0 & 0 & 0 & 0 & 0\\
4 & 0 & 0 & 0 & 0 & 0 & 0 & 0 & 0 & 0 & 0 & 0 & 0 & 0 & 0\\
0 & 0 & 0 & 0 & 0 & 0 & 4 & 0 & 0 & 0 & 0 & 0 & 0 & 0 & 0\\
0 & 0 & 0 & 0 & 0 & 0 & 0 & 4 & 0 & 0 & 0 & 0 & 0 & 0 & 0\\
0 & 0 & 0 & 0 & 0 & 0 & 0 & 0 & 4 & 0 & 0 & 0 & 0 & 0 & 0\\
0 & 0 & 0 & 0 & 0 & 0 & 0 & 0 &  0 & 4 & 0 & 0 & 0 & 0 & 0\\
0 & 0 & 0 & 0 & 0 & 0 & 0 & 0 & 0 & 0 & 4 & 0 & 0 & 0 & 0\\
0 & 0 & 0 & 0 & 0 & 4 & 0 & 0 & 0 & 0 & 0 & 0 & 0 & 0 & 0\\
0 & 0 & 0 & 0 & 0 & 0 & 0 & 0 & 0 & 0 & 0 & 0 & 0 & 0 &4\\
0 & 0 & 0 & 0 & 0 & 0 & 0 & 0 & 0 & 0 & 0 & 4 & 0 &0 & 0\\
0 & 0 & 0 & 0 & 0 & 0 & 0 & 0 & 0 & 0 & 0 & 0 & 0 & 0 & 2\\
0 & 0 & 0 & 0 & 0 & 0 & 0 & 0 & 0 & 0 & 0 & 0 & 4 & 0 & 0\\
\end{pmatrix} .
\]
The corresponding principal-sunflower max-eigenvector of $A$

\[
x_2=\begin{pmatrix}
\frac{1}{4}&\frac{1}{4}&\frac{1}{4}&\frac{1}{4}&\frac{1}{4}&1&1&1&1&1&1&\frac{1}{4}&\frac{1}{4}&\frac{1}{8}&\frac{1}{4}
\end{pmatrix}^T.
\]
By taking a critical circuit $(12,15), (15,13), (13,12)$ we obtain
\[ 
A_3^{*}=\begin{pmatrix}
0 & 4 & 0 & 0 & 0 & 0 & 0 & 0 & 0 & 0 & 0 & 0 & 0 & 0 & 0\\
0 & 0 & 4 & 0 & 0 & 0 & 0 & 0 & 0 & 0 & 0 & 0 & 0 & 0 & 0\\
0 & 0 & 0 & 0 & 4 & 0 & 0 & 0 & 0 & 0 & 0 & 0 & 0 & 0 & 0\\
0 & 0 & 0 & 0 & 0 & 0 & 1 & 0 & 0 & 0 & 0 & 0 & 0 & 0 & 0\\
4 & 0 & 0 & 0 & 0 & 0 & 0 & 0 & 0 & 0 & 0 & 0 & 0 & 0 & 0\\
0 & 0 & 0 & 0 & 0 & 0 & 4 & 0 & 0 & 0 & 0 & 0 & 0 & 0 & 0\\
0 & 0 & 0 & 0 & 0 & 0 & 0 & 4 & 0 & 0 & 0 & 0 & 0 & 0 & 0\\
0 & 0 & 0 & 0 & 0 & 0 & 0 & 0 & 4 & 0 & 0 & 0 & 0 & 0 & 0\\
0 & 0 & 0 & 0 & 0 & 0 & 0 & 0 & 0 & 4 & 0 & 0 & 0 & 0 & 0\\
0 & 0 & 0 & 0 & 0 & 0 & 0 & 0 & 0 & 0 & 4 & 0 & 0 & 0 & 0\\
0 & 0 & 0 & 0 & 0 & 4 & 0 & 0 & 0 & 0 & 0 & 0 & 0 & 0 & 0\\
0 & 0 & 0 & 0 & 0 & 0 & 0 & 0 & 0 & 0 & 0 & 0 & 0 & 0 & 4\\
0 & 0 & 0 & 0 & 0 & 0 & 0 & 0 & 0 & 0 & 0 & 4 & 0 & 0 & 0\\
0 & 0 & 0 & 0 & 0 & 0 & 0 & 0 & 0 & 0 & 0 & 0 & 0 & 0 & 2\\
0 & 0 & 0 & 0 & 0 & 0 & 0 & 0 & 0 & 0 & 0 & 0 & 4 & 0 & 0\\
\end{pmatrix} .
\]
The corresponding principal-sunflower max-eigenvector of $A$
\[
x_3=\begin{pmatrix}
\frac{1}{2}&\frac{1}{2}&\frac{1}{2}&\frac{3}{16}&\frac{1}{2}&\frac{3}{4}&\frac{3}{4}&\frac{3}{4}&\frac{3}{4}&\frac{3}{4}&\frac{3}{4}&1&1&\frac{1}{2}&1
\end{pmatrix}^T.
\]
The set $\{x_1 ,x_2, x_3\}$ is a basis of $V_{\oplus} (A,4)$. 
For $A^{'}=\frac{1}{4}A$ and after quite some calculation one can calculate $\Delta(A^{'})$ (see Appendix). The first three columns and column five of $\Delta(A^{'})$ are multiples of $x_1$, columns 6, 7, 8, 9, 10, 11 of $\Delta(A^{'})$ are multiples of $x_2$  and columns 12, 13, 15 of $\Delta(A^{'})$ are multiples of $x_3$.
\end{exam}
\medskip

\noindent{\bf  Acknowledgements}\label{S:acknowledgement}

The authors thank the reviewers and editor for their very useful suggestions and corrections, which considerably improved the quality of the article. 

This work was supported in part by the Department of Mathematical Sciences at
Isfahan University of Technology, Iran and  by  the Slovenian Research Agency (grants P1-0222, J1-8133, J1-8155, J2-2512 and N1-0071).

\section*{Declarations}

The authors declare that they have no conflict of interests.

\begin{appendices}

\section{}\label{secA1}
All circuits of matrix A in Example \ref{ESM} are given in \cite{BCOQ}. So we have 
\[a_{13}a_{32}a_{21}=2\times 2\times 2\]
\[a_{13}a_{39}a_{92}a_{21}=2\times 2\times 2\times 2\]
\[a_{32}a_{28}a_{81}a_{13}=2\times 4\times 2\times 2\]
\[a_{32}a_{28}a_{81}a_{17}a_{73}=2\times 4\times 2\times 2\times 2\]
\[a_{32}a_{28}a_{86}a_{69}a_{94}a_{47}a_{73}=2\times 4\times 4\times 2\times 4\times 2\times 2\]
\[a_{17}a_{73}a_{32}a_{21}=2\times 2\times 2\times 2\]
\[a_{45}a_{56}a_{64}=4\times 2\times 2\]
\[a_{45}a_{56}a_{69}a_{94}=4\times 4\times 2\times 4\]
\[a_{47}a_{75}a_{58}a_{86}a_{64}=2\times 4\times 4\times 4\times 4\]
\[a_{58}a_{86}a_{64}a_{45}=4\times 4\times 4\times 4\]
\[a_{39}a_{92}a_{28}a_{81}a_{13}=2\times 2\times 4\times 2\times 2\]
\[a_{75}a_{56}a_{69}a_{92}a_{28}a_{81}a_{17}=4\times 4\times 2\times 2\times 4\times 2\times 2\]
\[a_{75}a_{58}a_{86}a_{69}a_{94}a_{47}=4\times 4\times 4\times 2\times 2\times 2\]

Therefore $\mu(A)=\sqrt[4]{a_{58}a_{86}a_{64}a_{45}}=4$ and
\[
M_1=a_{58}a_{86}a_{64}a_{45}>M_2=a_{47}a_{75}a_{58}a_{86}a_{64}>M_3=a_{45}a_{56}a_{69}a_{94}\]\[
>M_4=a_{75}a_{58}a_{86}a_{69}a_{94}a_{47}>M_5=a_{75}a_{56}a_{69}a_{92}a_{28}a_{81}a_{17}\]\[=M_6=a_{32}a_{28}a_{86}a_{69}a_{94}a_{47}a_{73}
\] which implies
\[
A^*=\begin{pmatrix}
 0 & 0 & 0 & 0 &0 &0 &2  & 0 & 0 \\ 
 0& 0  & 0 & 0 &0 & 0 & 0 &4 & 0 \\ 
 0& 2  & 0 & 0 &0 & 0 &0 & 0 & 0 \\ 
 0& 0 & 0 & 0 & 4 &0 & 0  &0 & 0 \\ 
 0& 0 & 0 & 0 & 0& 0 & 0 & 4 & 0 \\ 
 0& 0 & 0 & 4  & 0 & 0 &0  & 0 & 0\\ 
  0& 0 & 0 & 0 & 4 & 0& 0& 0 & 0 \\ 
0 & 0 &0  & 0 & 0 & 4& 0& 0& 0 \\ 
 0& 0 & 0 & 4 &0 & 0& 0 & 0 & 0 \\ 
\end{pmatrix}
\]
Thus max eigenvector is $x=(\frac{1}{2}, 1, \frac{1}{2}, 1, 1, 1, 1, 1, 1)^T$

\bigskip

\noindent{\bf Karp's algorithm in max-plus algebra:} 

1)Choose arbitrary $j\in \{1,\dots,n\}$ and set $x(0)=e_j$, where $e_j$ is the standard vector in max-plus .

2)Compute $x(k)$ for $k=0,\dots,n$.

3)Compute a max-plus eigenvalue
\[\lambda=\underset{i=1,\dots,n}{\max} ~~\underset{k=0,\dots,n-1}{\min}\dfrac{x_i(n)-x_i(k)}{n-k}\]

 For matrix $A$ in our example, first we need to convert it in max-plus algebra
 as follows
 \[A_1=\begin{pmatrix}
 \ln1 & -\infty & \ln 2 & -\infty &-\infty &-\infty &\ln 2  & -\infty & -\infty \\ 
 \ln 2 & \ln 1  &  -\infty &  -\infty & -\infty &  -\infty &  -\infty &\ln 4 &  -\infty \\ 
 -\infty & \ln 2  & \ln  & -\infty & -\infty & -\infty & -\infty & -\infty & \ln 2 \\ 
 -\infty& -\infty & -\infty & \ln 1 & \ln 4 &-\infty & \ln 2  &-\infty & -\infty \\ 
 -\infty & -\infty & -\infty & -\infty & \ln 1& \ln 4 & -\infty & \ln 4 & -\infty \\ 
 -\infty & -\infty & -\infty & \ln4  & -\infty & \ln1 & -\infty   & -\infty & \ln 2\\ 
 -\infty & -\infty & \ln 2 & -\infty & \ln 4 & -\infty & \ln 2& -\infty  & -\infty \\ 
\ln 2 & -\infty & -\infty   & -\infty & -\infty & \ln 4&  -\infty & \ln 4& -\infty \\ 
 -\infty & \ln 2 &  -\infty & \ln 4 & -\infty &  -\infty &  -\infty & -\infty & \ln 2 \\ 
\end{pmatrix}
\]
Consider $x(0)=(0,-\infty, -\infty,-\infty,-\infty,-\infty,-\infty,-\infty,-\infty)^T$
\[x(1)=A_1\otimes x(0)=
\begin{pmatrix}
 \ln1 & -\infty & \ln2 & -\infty &-\infty &-\infty &\ln2  & -\infty & -\infty \\ 
 \ln 2 & \ln 1  &  -\infty &  -\infty & -\infty &  -\infty &  -\infty &\ln 4 &  -\infty \\ 
 -\infty & \ln 2  & \ln  & -\infty & -\infty & -\infty & -\infty & -\infty & \ln 2 \\ 
 -\infty& -\infty & -\infty & \ln 1 & \ln 4 &-\infty & \ln 2  &-\infty & -\infty \\ 
 -\infty & -\infty & -\infty & -\infty & \ln 1& \ln 4 & -\infty & \ln 4 & -\infty \\ 
 -\infty & -\infty & -\infty & \ln4  & -\infty & \ln1 & -\infty   & -\infty & \ln 2\\ 
 -\infty & -\infty & \ln 2 & -\infty & \ln 4 & -\infty & \ln 2& -\infty  & -\infty \\ 
\ln 2 & -\infty & -\infty   & -\infty & -\infty & \ln 4&  -\infty & \ln 4& -\infty \\ 
 -\infty & \ln 2 &  -\infty & \ln 4 & -\infty &  -\infty &  -\infty & -\infty & \ln 2 \\ 
\end{pmatrix}
\otimes
\begin{bmatrix}
0 \\ -\infty \\ -\infty \\ -\infty  \\-\infty \\ -\infty \\ -\infty \\ -\infty \\ -\infty 
\end{bmatrix}
\]\[=
\begin{bmatrix}
\ln1 \\ \ln2 \\ -\infty \\ -\infty  \\-\infty \\ -\infty \\ -\infty \\ \ln2\\ -\infty 
\end{bmatrix}
\]

\[x(2)=A_1\otimes x(1)=
\begin{bmatrix}
\ln1 \\ \ln8 \\ \ln4\\-\infty \\ \ln8  \\ -\infty \\ -\infty \\  \ln8\\ \ln4 
\end{bmatrix}.
\]
With similar calculation we have
\[
x(0)=\begin{bmatrix}
0 \\ -\infty \\ -\infty \\ -\infty  \\-\infty \\ -\infty \\ -\infty \\ -\infty \\ -\infty 
\end{bmatrix},
x(1)=\begin{bmatrix}
\ln1 \\ \ln2 \\ -\infty \\ -\infty  \\-\infty \\ -\infty \\ -\infty \\ \ln2\\ -\infty 
\end{bmatrix}
,
x(2)=
\begin{bmatrix}
\ln1 \\ \ln8 \\ \ln4\\-\infty \\ \ln8  \\ -\infty \\ -\infty \\ \ln8\\ \ln4 
\end{bmatrix}
,
x(3)=\begin{bmatrix}
\ln8  \\ \ln32\\ \ln16 \\ \ln32  \\ \ln32 \\ \ln8 \\ \ln32\\ \ln32\\ \ln16
\end{bmatrix}
,
x(4)=\begin{bmatrix}
\ln16  \\ \ln128\\ \ln64 \\ \ln128  \\  \ln128\\ \ln128 \\ \ln128\\ \ln128\\ \ln128\\ \ln128
\end{bmatrix}
\]
\[
x(5)=\begin{bmatrix}
\ln 256 \\ \ln 512\\ \ln256 \\ \ln512 \\ \ln512\\ \ln512 \\ \ln512\\ \ln512\\ \ln512\\ \ln512
\end{bmatrix}
,
x(6)=\begin{bmatrix}
\ln1024 \\ \ln2048\\ \ln1024 \\ \ln2048 \\  \ln2048\\  \ln2048 \\  \ln2048 \\ \ln2048 \\ \ln2048 \\ \ln2048 
\end{bmatrix}
,
x(7)=\begin{bmatrix}
\ln4096 \\ \ln8192\\ \ln4096 \\ \ln8192 \\  \ln8192\\ \ln8192 \\ \ln8192 \\ \ln8192 \\ \ln8192 \\ \ln8192
\end{bmatrix}
\]\[
x(8)=\begin{bmatrix}
\ln16384 \\ \ln32768\\ \ln16384 \\ \ln32768 \\ \ln32768\\ \ln32768 \\ \ln32768\\ \ln32768 \\ \ln32768 \\ \ln32768
\end{bmatrix}
,
x(9)=\begin{bmatrix}
\ln32768 \\ \ln131072\\ \ln32768 \\ \ln131072 \\ \ln131072\\ \ln131072 \\ \ln131072\\ \ln131072 \\ \ln131072 \\ \ln131072
\end{bmatrix}
\]

\[\frac{x_i(9)-x_i(k)}{9-k},~~~
\mbox{for}~~i=1,2,3,\dots,9~~ \mbox{and}~~ k=0,1,2,\dots,8.\] 
If $i=1$, then
\[\min\left\{\dfrac{\ln32768-\ln1}{9-0},\dfrac{\ln32768-\ln1}{9-1},\dots,\dfrac{\ln32768-\ln16384}{9-8}\right\}=\ln2.\]

If $i=2$, then 

\[ \min \left\{\dfrac{\ln131072}{9-0},\dfrac{\ln65536}{8},\dots,\dfrac{\ln4}{1}\right\}=\dfrac{\ln4}{1}.\]

If $i=3$, then 

\[ \min \left\{ \dfrac{\ln3768}{9},\dfrac{\ln3768}{8},\dots,\dfrac{\ln2}{1}\right\}=\dfrac{\ln2}{1}.\]

 If $i=4$, then

\[ \min \left\{\dfrac{\ln131072}{9},\dfrac{\ln131072}{8},\dots,\dfrac{\ln16}{1} \right\}=\dfrac{\ln4}{1}. \]

If  $i=5$, then
\[ \min \left\{\dfrac{\ln131072}{9},\dots,\dfrac{\ln4}{1}\right\}=\dfrac{\ln4}{1}\]

With similar calculation for $i=6, 7, 8, 9$ the minimum will be equal to $\ln 4$.

Therefore
\[\max\{\ln2, \ln4,\ln2, \ln4, \ln4, \ln4, \ln4, \ln4, \ln4\}=\ln4,\]
which implies $\mu(A)=e^{\ln4}=4$.

By observing that $x(k + 1) =( \ln 4) . x(k)$ for $k \ge 4$, we can conclude that there is
no need to continue our calculations further. After this step, all subsequent values
of $x(k)$ will be equal to $(\ln 4)^{k-4} \otimes x(4)$ for $5 \le k \le 9$. Consequently, 
$x(4)$ is the corresponding maximum eigenvector for $A_1$.
\bigskip

\noindent{\bf Power method:}
\[
A=\begin{pmatrix}
 1 & 0 & 2 & 0 &0 &0 &2  & 0 & 0 \\ 
 2 & 1  & 0 & 0 &0 & 0 & 0 &4 & 0 \\ 
 0& 2  & 1 & 0 &0 & 0 &0 & 0 & 2 \\ 
 0& 0 & 0 & 1 & 4 &0 & 2  &0 & 0 \\ 
 0& 0 & 0 & 0 & 1& 4 & 0 & 4 & 0 \\ 
 0& 0 & 0 & 4  & 0 & 1 &0  & 0 & 2\\ 
  0& 0 & 2 & 0 & 4 & 0& 2& 0 & 0 \\ 
2 & 0 &0  & 0 & 0 & 4& 0& 4& 0 \\ 
 0& 2 & 0 & 4 &0 & 0& 0 & 0 & 2 \\ 
\end{pmatrix}
\]

 If $j=1$, we choose the vector $x(0)=(1,0,0,0,0,0,0,0,0)$.
Now we calculate the vectors $x(k)$ based on the following relationship.
\[x(k+1)=A\otimes x(k)\]

If $k=0$ then 
\[x(0+1)=A\otimes x(0)\]
\[x(1)= \begin{pmatrix}
 1 & 0 & 2 & 0 &0 &0 &2  & 0 & 0 \\ 
 2 & 1  & 0 & 0 &0 & 0 & 0 &4 & 0 \\ 
 0& 2  & 1 & 0 &0 & 0 &0 & 0 & 2 \\ 
 0& 0 & 0 & 1 & 4 &0 & 2  &0 & 0 \\ 
 0& 0 & 0 & 0 & 1& 4 & 0 & 4 & 0 \\ 
 0& 0 & 0 & 4  & 0 & 1 &0  & 0 & 2\\ 
  0& 0 & 2 & 0 & 4 & 0& 2& 0 & 0 \\ 
2 & 0 &0  & 0 & 0 & 4& 0& 4& 0 \\ 
 0& 2 & 0 & 4 &0 & 0& 0 & 0 & 2 \\ 
\end{pmatrix}
\otimes\begin{pmatrix}
1 \\
0\\0\\0\\0\\0\\0\\0\\0
\end{pmatrix}= \begin{pmatrix}
1 \\
2\\0\\0\\0\\0\\0\\2\\0
\end{pmatrix}
\]

If $k=1$ then 
\[x(1+1)=A\otimes x(1)\]
\[x(2)=  \begin{pmatrix}
 1 \\ 8\\4\\0\\8\\0\\0\\8\\4
\end{pmatrix}
\]

If $k=2$ then 
\[x(2+1)=A\otimes x(2)\]
\[x(3)= \begin{pmatrix}
 8 \\ 32\\16\\32\\32\\8\\32\\32\\16
\end{pmatrix}
\]

If $k=3$ then 
\[x(3+1)=A\otimes x(3)\]
\[x(4)= \begin{pmatrix}
2 \times 32\\ 4\times 32\\2 \times 32\\4 \times 32\\4 \times 32\\4 \times 32\\4 \times 32\\4 \times 32\\4 \times 32
\end{pmatrix}
\]

If $k=4$ then 
\[x(4+1)=A\otimes x(4)\]
\[x(5)=  \begin{pmatrix}
8 \times 32\\ 16\times 32\\8 \times 32\\16 \times 32\\16 \times 32\\16 \times 32\\16 \times 32\\16 \times 32\\16 \times 32
\end{pmatrix}
.\]
Therefore $x(5)=4\otimes x(4)$ which implies in power algorithm $c=4, p=5$ and $q=4$. Thus
\[
\mu(A)=\frac{c}{p-q}=\frac{4}{5-4}
\]
and the max-eigenvector is
\[
x=\bigoplus_{j=1}^{5-4}(\mu(A))^{5-4-j}\otimes x(4+j-1)=x(4).
\]
\section{}\label{secA2}

The matrix $\Delta(A^{'})$ from Example \ref{big} equals
\[
\Delta(A^{'})=I \oplus A \oplus \cdots \oplus A^{14} _{\otimes}=\]\[\begin{pmatrix}
 1 & 1 & 1 & \frac{1}{4} & 1 & \frac{1}{4} & \frac{1}{4} & \frac{1}{4} & \frac{1}{4} & \frac{1}{4} & \frac{1}{4} & \frac{1}{2} & \frac{1}{2} & \frac{1}{4} & \frac{1}{2} \\ \\
 1 & 1 & 1 & \frac{1}{4} & 1 & \frac{1}{4} & \frac{1}{4} & \frac{1}{4} & \frac{1}{4} & \frac{1}{4} & \frac{1}{4} & \frac{1}{2} & \frac{1}{2} & \frac{1}{4} & \frac{1}{2} \\ \\
 1 & 1 & 1 & \frac{1}{4} & 1 & \frac{1}{4} & \frac{1}{4} & \frac{1}{4} & \frac{1}{4} & \frac{1}{4} & \frac{1}{4} & \frac{1}{2} & \frac{1}{2} & \frac{1}{4} & \frac{1}{2} \\ \\
 \frac{1}{4} & \frac{1}{4} & \frac{1}{4} & 1 & \frac{1}{4} & \frac{1}{4} & \frac{1}{4} & \frac{1}{4} & \frac{1}{4} & \frac{1}{4} & \frac{1}{4} & \frac{3}{16} & \frac{3}{16} & \frac{1}{4} & \frac{3}{16} \\ \\
 1 & 1 & 1 & \frac{1}{4} & 1 & \frac{1}{4} & \frac{1}{4} & \frac{1}{4} & \frac{1}{4} & \frac{1}{4} & \frac{1}{4} & \frac{1}{2} & \frac{1}{2} & \frac{1}{4} & \frac{1}{2} \\ \\
 \frac{1}{2} & \frac{1}{2} & \frac{1}{2} & \frac{1}{4} & \frac{1}{2} & 1 & 1 & 1 & 1 & 1 & 1 & \frac{3}{4} & \frac{3}{4} & \frac{1}{4} & \frac{3}{4} \\ \\
 \frac{1}{2} & \frac{1}{2} & \frac{1}{2} & \frac{1}{4} & \frac{1}{2} & 1 & 1 & 1 & 1 & 1 & 1 & \frac{3}{4} & \frac{3}{4} & \frac{1}{4} & \frac{3}{4} \\ \\
 \frac{1}{2} & \frac{1}{2} & \frac{1}{2} & \frac{1}{4} & \frac{1}{2} & 1 & 1 & 1 & 1 & 1 & 1 & \frac{3}{4} & \frac{3}{4} & \frac{1}{4} & \frac{3}{4} \\ \\
 \frac{1}{2} & \frac{1}{2} & \frac{1}{2} & \frac{1}{4} & \frac{1}{2} & 1 & 1 & 1 & 1 & 1 & 1 & \frac{3}{4} & \frac{3}{4} & \frac{1}{4} & \frac{3}{4} \\ \\
 \frac{1}{2} & \frac{1}{2} & \frac{1}{2} & \frac{1}{4} & \frac{1}{2} & 1 & 1 & 1 & 1 & 1 & 1 & \frac{3}{4} & \frac{3}{4} & \frac{1}{4} & \frac{3}{4} \\ \\
 \frac{1}{2} & \frac{1}{2} & \frac{1}{2} & \frac{1}{4} & \frac{1}{2} & 1 & 1 & 1 & 1 & 1 & 1 & \frac{3}{4} & \frac{3}{4} & \frac{1}{4} & \frac{3}{4} \\ \\
 \frac{1}{2} & \frac{1}{2} & \frac{1}{2} & \frac{1}{4} & \frac{1}{2} & \frac{1}{4} & \frac{1}{4} & \frac{1}{4} & \frac{1}{4} & \frac{1}{4} & \frac{1}{4} & 1 & 1 & \frac{1}{4} & 1 \\ \\
 \frac{1}{2} & \frac{1}{2} & \frac{1}{2} & \frac{1}{4} & \frac{1}{2} & \frac{1}{4} & \frac{1}{4} & \frac{1}{4} & \frac{1}{4} & \frac{1}{4} & \frac{1}{4} & 1 & 1 & \frac{1}{4} & 1 \\ \\
 \frac{1}{4} & \frac{1}{4} & \frac{1}{4} & \frac{1}{4} & \frac{1}{4} & \frac{1}{8} & \frac{1}{8} & \frac{1}{8} & \frac{1}{8} & \frac{1}{8} & \frac{1}{8} & \frac{1}{2} & \frac{1}{2} & 1 & \frac{1}{2} \\ \\
 \frac{1}{2} & \frac{1}{2} & \frac{1}{2} & \frac{1}{4} & \frac{1}{2} & \frac{1}{4} & \frac{1}{4} & \frac{1}{4} & \frac{1}{4} & \frac{1}{4} & \frac{1}{4} & 1 & 1 & \frac{1}{4} & 1 \\
\end{pmatrix}.
\]

\end{appendices}

\end{document}